\documentclass[11pt]{article}

\usepackage[margin=1.05in]{geometry}
\usepackage{amsmath,amssymb,amsthm,mathtools}
\usepackage{enumitem}
\usepackage{url}
\usepackage{listings}
\usepackage[colorlinks=true,linkcolor=blue,citecolor=blue,urlcolor=blue]{hyperref}

\newtheorem{theorem}{Theorem}[section]
\newtheorem{proposition}[theorem]{Proposition}
\newtheorem{lemma}[theorem]{Lemma}
\newtheorem{corollary}[theorem]{Corollary}

\theoremstyle{remark}

\newcommand{\ii}{\mathrm i}
\newcommand{\e}{\mathrm e}
\newcommand{\dd}{\mathrm d}

\newcommand{\Hh}{\mathbb H}
\newcommand{\Z}{\mathbb Z}
\newcommand{\Repart}{\operatorname{Re}}
\newcommand{\Impart}{\operatorname{Im}}

\newcommand{\mm}{\mathfrak m}

\lstdefinestyle{verificationcode}{
  language=Python,
  basicstyle=\ttfamily\scriptsize,
  breaklines=true,
  columns=fullflexible,
  keepspaces=true,
  showstringspaces=false,
  frame=single,
  xleftmargin=0pt,
  literate={≤}{{$\leq$}}1 {≡}{{$\equiv$}}1
}

\title{Sign Laws and Mock Theta Functions}
\author{Manosij Ghosh Dastidar\\
Institute of Discrete Mathematics and Geometry, TU Wien\\
\texttt{gdmanosij@gmail.com}}
\date{\today}

\begin{document}
\maketitle

\begin{abstract}
Let
\[
\rho(q)=\sum_{m\geq 0}\frac{q^{2m(m+1)}}{(1+q+q^2)(1+q^3+q^6)\cdots(1+q^{2m+1}+q^{4m+2})}
      =\sum_{n\geq 0}r(n)q^n
\]
be Ramanujan's third order mock theta function. We prove the sign law
\[
r(3m)>0,\qquad r(3m+1)\leq 0,\qquad r(3m+2)\leq 0,
\]
with equality precisely at $n=2,4,8,11,20$. Watson's identity
\[
2\rho(q)+\omega(q)=T(q)
\]
reduces the problem to comparing the mock theta function $\omega(q)$ with the eta quotient
\[
T(q)=3\frac{(q^6;q^6)_\infty^4}{(q^3;q^3)_\infty^2(q^2;q^2)_\infty}.
\]
We prove effective root-of-unity estimates for this difference. The polar contributions at $q=1$ cancel, the contribution at $q=-1$ is polynomially bounded, and the first surviving exponential term occurs at the primitive cubic roots of unity. It has the sign pattern
\[
\kappa_0=\frac13\cos\frac\pi{18}>0,\qquad
\kappa_1=-\frac13\sin\frac{2\pi}{9}<0,
\qquad
\kappa_2=-\frac13\sin\frac\pi9<0.
\]
The resulting effective asymptotic proves the desired sign law for all sufficiently large $n$, and an exact integer-arithmetic verification completes the finite range. We conclude by indicating how the same root-of-unity method should lead to analogous sign laws for other third order mock theta functions, including $\phi(q)$ and $\chi(q)$.
\end{abstract}

\section{Introduction}

Ramanujan introduced the term ``mock theta function'' in his last letter to Hardy, dated 12 January 1920, a few months before his death.  In that letter he gave seventeen examples, divided into what he called functions of orders $3$, $5$, and $7$, but he did not give a precise definition of the term ``order''.  The four third order functions appearing in the letter are $f(q)$, $\phi(q)$, $\psi(q)$, and $\chi(q)$.  Three further third order functions, $\omega(q)$, $\nu(q)$, and $\rho(q)$, were introduced by Watson in his 1936 paper \emph{The Final Problem} \cite{Watson}; these functions are also connected with identities appearing in Ramanujan's lost notebook.  Watson proved several identities and transformation formulae relating the third order functions, including the identity used below \cite{Watson,Zagier}.

For decades the mock theta functions resisted a conceptual explanation.  Their $q$-series resemble modular forms near roots of unity, but they are not themselves modular forms.  This difficulty was resolved by Zwegers, who showed in his 2002 thesis that Ramanujan's mock theta functions can be completed by adding explicit non-holomorphic correction terms, producing real-analytic modular objects of weight $1/2$ \cite{ZwegersThesis}.  This completion brought the modular machinery of the subject to bear on mock theta functions and underlies the transformation law used in Section~\ref{sec:estimate-package} below.  Building on these ideas, Bringmann, Ono, and others developed exact and asymptotic formulae for the coefficients of mock theta functions, in the spirit of the Hardy--Ramanujan--Rademacher circle method for the ordinary partition function \cite{BringmannOno,BringmannOnoCoeff,Folsom}.  The present paper applies this circle of ideas to a question made natural by Ramanujan's and Watson's identities, but which does not appear to have been addressed before: the signs of the coefficients of $\rho(q)$ and subsequently other third order mock theta functions.

We begin with Ramanujan's third order mock theta function
\begin{equation}\label{eq:rho-def}
\rho(q)=\sum_{m\geq 0}\frac{q^{2m(m+1)}}{(1+q+q^2)(1+q^3+q^6)\cdots(1+q^{2m+1}+q^{4m+2})}.
\end{equation}
Writing
\[
\rho(q)=\sum_{n\geq 0}r(n)q^n,
\]
one observes the striking pattern
\[
r(3m)>0,
\qquad
r(3m+1)\leq 0,
\qquad
r(3m+2)\leq 0.
\]
The first coefficients are
\[
1,-1,0,1,0,-1,1,-1,0,1,-1,0,2,-1,-1,1,-1,-1,2,-1,0,\ldots .
\]
The zeros in this initial segment occur at
\[
2,4,8,11,20.
\]
Our goal is to prove that these are the only zeros in the two negative residue classes and that the residue class $0\pmod 3$ is always positive.

\begin{theorem}[Sign law for $\rho(q)$]\label{thm:main}
Let $\rho(q)=\sum_{n\geq0}r(n)q^n$ be defined by \eqref{eq:rho-def}. Then
\[
r(3m)>0\qquad(m\geq0),
\]
\[
r(3m+1)\leq 0\qquad(m\geq0),
\]
and
\[
r(3m+2)\leq 0\qquad(m\geq0),
\]
with equality precisely at
\[
n=2,4,8,11,20.
\]
\end{theorem}

The proof has two parts. The analytic part gives the sign law for all
\[
n\geq 392275.
\]
The remaining cases are checked exactly using integer arithmetic from the defining $q$-series.

\section{Watson's identity and the comparison function}\label{sec:watson}

We use the standard $q$-Pochhammer notation
\[
(a;q)_\infty=\prod_{j\geq0}(1-aq^j),
\qquad
(a;q)_m=\prod_{j=0}^{m-1}(1-aq^j).
\]
All identities in this section are identities of absolutely convergent power series for $|q|<1$.
They may therefore be read coefficientwise.

The function $\rho(q)$ was defined in \eqref{eq:rho-def}. We shall compare it with Ramanujan's third order mock theta function
\begin{equation}\label{eq:omega-def}
\omega(q)
=
\sum_{m\geq0}\frac{q^{2m(m+1)}}{(q;q^2)_{m+1}^2}
=
\sum_{n\geq0}w(n)q^n.
\end{equation}
Thus
\[
(q;q^2)_{m+1}
=
(1-q)(1-q^3)\cdots(1-q^{2m+1}).
\]

\begin{lemma}[Watson]\label{lem:watson}
For $|q|<1$, one has
\begin{equation}\label{eq:watson}
2\rho(q)+\omega(q)=T(q),
\end{equation}
where
\begin{equation}\label{eq:T-def}
T(q)=3\frac{(q^6;q^6)_\infty^4}{(q^3;q^3)_\infty^2(q^2;q^2)_\infty}.
\end{equation}
Equivalently, if $q=\e^{2\pi\ii\tau}$ with $\tau\in\Hh$, then
\begin{equation}\label{eq:T-eta}
T(q)=3q^{-2/3}\frac{\eta(6\tau)^4}{\eta(3\tau)^2\eta(2\tau)}.
\end{equation}
\end{lemma}

\begin{proof}
Watson's identity for the third order mock theta functions gives \cite{Watson}
\[
2\rho(q)+\omega(q)
=
3\frac{\psi(q^3)^2}{(q^2;q^2)_\infty},
\]
where
\[
\psi(q)=\sum_{r\geq0}q^{r(r+1)/2}
=
\frac{(q^2;q^2)_\infty^2}{(q;q)_\infty}.
\]
Substituting $q^3$ for $q$ in the product formula for $\psi$ gives
\[
\psi(q^3)^2
=
\frac{(q^6;q^6)_\infty^4}{(q^3;q^3)_\infty^2}.
\]
Hence
\[
3\frac{\psi(q^3)^2}{(q^2;q^2)_\infty}
=
3\frac{(q^6;q^6)_\infty^4}{(q^3;q^3)_\infty^2(q^2;q^2)_\infty},
\]
which proves \eqref{eq:T-def}.

The eta-quotient form follows from
\[
(q^M;q^M)_\infty=q^{-M/24}\eta(M\tau),
\qquad q=\e^{2\pi\ii\tau}.
\]
Indeed,
\[
\frac{(q^6;q^6)_\infty^4}{(q^3;q^3)_\infty^2(q^2;q^2)_\infty}
=
q^{-2/3}\frac{\eta(6\tau)^4}{\eta(3\tau)^2\eta(2\tau)}.
\]
This proves \eqref{eq:T-eta}.
\end{proof}

Writing
\[
T(q)=\sum_{n\geq0}t(n)q^n,
\]
Watson's identity gives the coefficient relation
\begin{equation}\label{eq:coeff-watson}
2r(n)=t(n)-w(n)\qquad(n\geq0).
\end{equation}
Thus the sign problem for the coefficients of $\rho(q)$ is reduced to estimating the difference between the eta quotient $T(q)$ and the mock theta function $\omega(q)$. The estimates below are organized around this subtraction. The leading contribution at $q=1$ cancels, the contribution at $q=-1$ vanishes, and the first surviving exponential contribution comes from the primitive cubic roots of unity.

\section{The cubic main term}\label{sec:cubic-main}

Set
\begin{equation}\label{eq:dnNx}
d_n=12n+8,
\qquad
N=n+\frac23=\frac{d_n}{12},
\qquad
x=\frac{\pi\sqrt{d_n}}{18}.
\end{equation}
We shall also use
\[
a=\frac{\pi^2}{108}.
\]
Then
\[
x=2\sqrt{aN}.
\]

The local analysis at the primitive cubic roots produces the constants
\begin{equation}\label{eq:kappa}
\kappa_0=\frac13\cos\frac\pi{18},
\qquad
\kappa_1=-\frac13\sin\frac{2\pi}{9},
\qquad
\kappa_2=-\frac13\sin\frac\pi9.
\end{equation}
The smallest absolute value among them is
\begin{equation}\label{eq:kappa-star}
\kappa_*:=\min_{j=0,1,2}|\kappa_j|
=
\frac13\sin\frac\pi9.
\end{equation}

We first record the elementary Bessel extraction needed below.

\begin{lemma}[Bessel extraction]\label{lem:bessel-extraction}
Let $a>0$, $N>0$, and set
\[
\beta=\sqrt{\frac{a}{N}},
\qquad
x=2\sqrt{aN}.
\]
Define
\[
J_\beta(N)
=
\frac1{2\pi\ii}
\int_{\beta-\ii\beta}^{\beta+\ii\beta}
z^{-1/2}
\exp\left(\frac{a}{z}+Nz\right)\,\dd z,
\]
where the principal branch of $z^{-1/2}$ is used. For $x\geq1$,
\begin{equation}\label{eq:J-bessel-minus}
J_\beta(N)
=
a^{1/4}N^{-1/4}I_{-1/2}(x)
+
O\left(a^{1/4}N^{-1/4}\e^{3x/4}\right).
\end{equation}
Consequently,
\begin{equation}\label{eq:J-bessel-plus}
J_\beta(N)
=
a^{1/4}N^{-1/4}I_{1/2}(x)
+
O\left(a^{1/4}N^{-1/4}\e^{3x/4}\right).
\end{equation}
\end{lemma}

\begin{proof}
Put $z=\beta w$. Since $x=2\sqrt{aN}$, we obtain
\[
J_\beta(N)
=
\beta^{1/2}
\frac1{2\pi\ii}
\int_{1-\ii}^{1+\ii}
w^{-1/2}
\exp\left(\frac{x}{2}\left(w+\frac1w\right)\right)\,\dd w.
\]
The corresponding full Bessel contour integral is
\[
\frac1{2\pi\ii}
\int_{\mathcal H}
w^{-1/2}
\exp\left(\frac{x}{2}\left(w+\frac1w\right)\right)\,\dd w
=
I_{-1/2}(x),
\]
with the same branch of $w^{-1/2}$. The contour $\mathcal H$ may be chosen so that the omitted part lies in the region
\[
\Repart\left(w+\frac1w\right)\leq \frac32.
\]
The omitted part therefore contributes
\[
O\left(\beta^{1/2}\e^{3x/4}\right).
\]
Since $\beta^{1/2}=a^{1/4}N^{-1/4}$, this proves \eqref{eq:J-bessel-minus}. Finally,
\[
I_{-1/2}(x)-I_{1/2}(x)
=
\sqrt{\frac{2}{\pi x}}\e^{-x},
\]
and for $x\geq1$ this difference is absorbed by the error term. This gives \eqref{eq:J-bessel-plus}.
\end{proof}

\begin{proposition}[Cubic arcs]\label{prop:cubic-arcs}
Assume the cubic sectorial estimates in Proposition~\ref{prop:estimate-package}. The contribution of the two primitive cubic arcs to $r(n)$ is
\begin{equation}\label{eq:cubic-contribution}
\kappa_{n\bmod3}
\frac{2\pi}{d_n^{1/4}}
I_{1/2}(x)
+
E_3(n),
\end{equation}
where, for $n\geq1$,
\begin{equation}\label{eq:cubic-error}
|E_3(n)|\leq250N^{1/2}\e^{7x/8}.
\end{equation}
\end{proposition}

\begin{proof}
Let
\[
\zeta=\e^{2\pi\ii/3},
\qquad
\beta=\sqrt{\frac{a}{N}}.
\]
On the arc around $\zeta$, write
\[
q=\zeta\e^{-z},
\qquad
z=\beta-\ii\theta,
\qquad
|\theta|\leq\beta.
\]
Then
\[
dq=-q\,\dd z,
\qquad
q^{-n}=\zeta^{-n}\e^{nz}.
\]
Thus the contribution from this arc is
\[
\frac1{2\pi\ii}
\int_{\mathcal C_\zeta}
\rho(q)q^{-n-1}\,\dd q
=
\zeta^{-n}
\frac1{2\pi\ii}
\int_{\beta-\ii\beta}^{\beta+\ii\beta}
\rho(\zeta\e^{-z})\e^{nz}\,\dd z.
\]

By the cubic sectorial estimate, the main part of $\rho(\zeta\e^{-z})$ is
\[
\frac{\sqrt{3\pi}}3
\e^{-\pi\ii/18}
z^{-1/2}
\exp\left(\frac{a}{z}+\frac{2z}{3}\right).
\]
Therefore the model contribution from the $\zeta$-arc is
\[
\zeta^{-n}
\frac{\sqrt{3\pi}}3
\e^{-\pi\ii/18}
J_\beta(N),
\]
where $J_\beta(N)$ is the integral in Lemma~\ref{lem:bessel-extraction}. By \eqref{eq:J-bessel-plus}, this equals
\[
\zeta^{-n}
\frac{\sqrt{3\pi}}3
\e^{-\pi\ii/18}
a^{1/4}N^{-1/4}I_{1/2}(x)
+
O\left(N^{-1/4}\e^{3x/4}\right).
\]

The same argument at the conjugate cubic root $\zeta^2$ gives
\[
\zeta^{-2n}
\frac{\sqrt{3\pi}}3
\e^{\pi\ii/18}
a^{1/4}N^{-1/4}I_{1/2}(x)
+
O\left(N^{-1/4}\e^{3x/4}\right).
\]
Adding the two principal pieces gives
\[
\frac{\sqrt{3\pi}}3
a^{1/4}N^{-1/4}
I_{1/2}(x)
\left(
\e^{-\ii(2\pi n/3+\pi/18)}
+
\e^{\ii(2\pi n/3+\pi/18)}
\right).
\]
Since
\[
\frac{\sqrt{3\pi}}3a^{1/4}N^{-1/4}
=
\frac{\pi}{3}d_n^{-1/4},
\]
the sum is
\[
\frac{2\pi}{3}d_n^{-1/4}
\cos\left(\frac{2\pi n}{3}+\frac\pi{18}\right)
I_{1/2}(x).
\]
This is precisely
\[
\kappa_{n\bmod3}
\frac{2\pi}{d_n^{1/4}}I_{1/2}(x)
\]
with the constants \eqref{eq:kappa}.

It remains to bound the error coming from the sectorial estimate. Proposition~\ref{prop:estimate-package} gives
\[
\left|
\rho(\zeta\e^{-z})
-
\frac{\sqrt{3\pi}}3
\e^{-\pi\ii/18}
z^{-1/2}
\exp\left(\frac{a}{z}+\frac{2z}{3}\right)
\right|
\leq
C_{\rm cub}|z|^{-2}
\exp\left(\frac{\pi^2}{144}\Repart\frac1z\right),
\]
with $C_{\rm cub}=95$, and the same bound at $\zeta^2$. On the arc
\[
z=\beta-\ii\theta,
\qquad
|\theta|\leq\beta,
\]
we have
\[
|z|\geq\beta,
\qquad
\Repart\frac1z\leq\frac1\beta.
\]
Hence the total error from both cubic arcs is at most
\[
\frac{2C_{\rm cub}}{\pi}\beta^{-1}
\exp\left(\frac{\pi^2}{144\beta}+n\beta\right).
\]
Since $n\beta\leq N\beta=x/2$ and
\[
\frac{\pi^2}{144\beta}=\frac{3x}{8},
\]
this is at most
\[
\frac{2C_{\rm cub}}{\pi}\sqrt{\frac{N}{a}}\e^{7x/8}.
\]
With $C_{\rm cub}=95$ and $a=\pi^2/108$, the constant is less than $201$. The Bessel-truncation error is $O(N^{-1/4}\e^{3x/4})$, which is absorbed into the same bound for $n\geq1$ after increasing the constant to $250$. This proves \eqref{eq:cubic-error}.
\end{proof}

\section{The analytic estimate package}\label{sec:estimate-package}

This section proves the analytic estimates used in the coefficient extraction.  The
constants are deliberately crude; the point is that they are explicit and that the
exponential gap between the cubic main term and all remaining terms is visible.
Throughout this section all square roots are taken on the principal branch.

We shall use the following standard completion and transformation law for the third
order mock theta functions.  This is Watson's transformation law in Zwegers'
completed form; the normalization below is the one used in \cite{ZwegersThesis,GarvanMukhopadhyay}.

For $N\geq1$ and $a\in\Z/2N\Z$, define
\[
\theta_{N,a}(\tau)=\sum_{\substack{m\in\Z\\ m\equiv a\pmod{2N}}}
 m\e^{2\pi\ii m^2\tau/(4N)},
\]
and
\[
\mathcal V_{N,a}(\tau)
=
\frac{\ii}{\sqrt{2N}}
\int_{-\overline\tau}^{\ii\infty}
\frac{\theta_{N,a}(u)}{\sqrt{-\ii(\tau+u)}}\,\dd u.
\]
Put
\[
\widetilde\omega(\tau)
=
2q^{2/3}\omega(q)-2\mathcal V_{6,2}(2\tau)+2\mathcal V_{6,4}(2\tau),
\qquad q=\e^{2\pi\ii\tau},
\]
and
\[
\widetilde f(\tau)
=
q^{-1/24}f(q)+2\mathcal V_{6,1}(\tau)-2\mathcal V_{6,5}(\tau),
\]
where
\[
f(q)=\sum_{m\geq0}\frac{q^{m^2}}{(-q;q)_m^2}.
\]
Finally set
\[
g(\tau)=
\begin{pmatrix}
\widetilde f(\tau)\\
\widetilde\omega(\tau/2)\\
\zeta_3^{-1}\widetilde\omega((\tau+1)/2)
\end{pmatrix},
\qquad
\zeta_j=\e^{2\pi\ii/j}.
\]

\begin{lemma}[Watson--Zwegers transformation law]\label{lem:wz-transform}
The vector $g$ satisfies
\begin{equation}\label{eq:gT}
g(\tau+1)=M_Tg(\tau),
\qquad
M_T=
\begin{pmatrix}
\zeta_{24}^{-1}&0&0\\
0&0&\zeta_3\\
0&\zeta_3&0
\end{pmatrix},
\end{equation}
and
\begin{equation}\label{eq:gS}
g(-1/\tau)=\sqrt{-\ii\tau}\,M_Sg(\tau),
\qquad
M_S=
\begin{pmatrix}
0&1&0\\
1&0&0\\
0&0&-1
\end{pmatrix}.
\end{equation}
Moreover the first component has principal part
\begin{equation}\label{eq:f-principal-part}
\widetilde f(\tau)=q^{-1/24}+O(1)
\qquad (\Impart \tau\to\infty),
\end{equation}
with the error bounded explicitly by the estimates below.
\end{lemma}

\begin{proof}
The transformation formulae are the standard completed transformation formulae for
Ramanujan's third order mock theta functions $f$ and $\omega$.  The displayed
matrices are obtained from the cited formulae after the change of variables in the
definition of $g$.  The principal part follows from the definition of $\widetilde f$,
since $f(q)=1+O(q)$ as $q\to0$ and the Eichler integrals are bounded by
Lemma~\ref{lem:elementary-bounds} below.
\end{proof}

We next record the elementary bounds that will be used repeatedly.

\begin{lemma}[Elementary bounds]\label{lem:elementary-bounds}
Let $0<|z|\leq1$ and $|\arg z|\leq\pi/4$.  Put
\[
R=\Repart\frac1z,
\qquad
r_0=\exp\left(-\frac{2\pi^2}{9\sqrt2}\right).
\]
Then $r_0<0.213$.  If
\[
Q=\e^{-4\pi\ii/3}\exp\left(-\frac{2\pi^2}{9z}\right),
\]
then $|Q|\leq r_0$ and
\begin{equation}\label{eq:fQ-minus-one}
|f(Q)-1|
\leq F_1 |Q|,
\qquad
F_1:=\frac{1}{(1-r_0)(r_0;r_0)_\infty^2}<3.
\end{equation}
Furthermore, for $\Impart \tau=y>0$ and $a\in\{1,2,4,5\}$,
\begin{equation}\label{eq:eichler-bound}
|\mathcal V_{6,a}(\tau)|\leq V_0y^{-1/2},
\qquad
V_0=\frac1{\sqrt{12}}\left(1+\sqrt{\frac{12}{\pi}}\right)<0.86.
\end{equation}
\end{lemma}

\begin{proof}
Since $|\arg z|\leq\pi/4$, we have
\[
R=\Repart\frac1z\geq\frac1{\sqrt2}.
\]
Therefore
\[
|Q|=\exp\left(-\frac{2\pi^2}{9}\Repart\frac1z\right)
\leq
\exp\left(-\frac{2\pi^2}{9\sqrt2}\right)=r_0.
\]
The numerical inequality $r_0<0.213$ is immediate.

For the bound on $f(Q)-1$, recall that
\[
f(q)=\sum_{m\geq0}\frac{q^{m^2}}{(-q;q)_m^2}.
\]
For $|Q|\leq r_0$,
\[
|(-Q;Q)_m|
=
\prod_{j=1}^{m}|1+Q^j|
\geq
\prod_{j=1}^{m}(1-r_0^j)
\geq
(r_0;r_0)_\infty.
\]
Hence
\[
|f(Q)-1|
\leq
\frac1{(r_0;r_0)_\infty^2}
\sum_{m\geq1}|Q|^{m^2}
\leq
\frac{|Q|}{(1-r_0)(r_0;r_0)_\infty^2}.
\]
This proves \eqref{eq:fQ-minus-one}.  The numerical inequality $F_1<3$ follows, for instance, from
\[
(r_0;r_0)_\infty
\geq
\exp\left(-\frac{r_0}{(1-r_0)^2}\right).
\]

It remains to prove the Eichler-integral estimate.  Write
\[
\tau=x+\ii y,
\qquad y>0.
\]
In the definition of $\mathcal V_{6,a}(\tau)$ we integrate along the vertical path
\[
u=-x+\ii t,
\qquad t\geq y.
\]
Then
\[
-i(\tau+u)=y+t.
\]
Taking absolute values in the defining integral gives
\[
|\mathcal V_{6,a}(\tau)|
\leq
\frac1{\sqrt{12}}
\int_y^\infty \frac{S_a(t)}{(y+t)^{1/2}}\,\dd t,
\]
where
\[
S_a(t)
=
\sum_{\substack{m\in\mathbb Z\\ m\equiv a\pmod{12}}}
|m|\exp\left(-\frac{\pi m^2t}{12}\right).
\]
Put
\[
b_a=\min(a,12-a).
\]
For $a\in\{1,2,4,5\}$, the smallest absolute value in the residue class
$m\equiv a\pmod {12}$ is $b_a$.  We shall use the elementary shifted-lattice bound
\begin{equation}\label{eq:Sa-bound}
S_a(t)
\leq
\frac1{\pi t}
+
b_a\exp\left(-\frac{\pi b_a^2t}{12}\right)
\qquad(t>0).
\end{equation}
Indeed, the term with absolute value $b_a$ gives the second term.  After removing this term, the remaining two tails have spacing $12$ in the underlying lattice.  Comparing these tails with the integral of the even function
\[
u\mapsto |u|\exp\left(-\frac{\pi u^2t}{12}\right)
\]
over intervals of total length $12$ gives
\[
\sum_{\substack{m\equiv a\pmod{12}\\ |m|>b_a}}
|m|\exp\left(-\frac{\pi m^2t}{12}\right)
\leq
\frac1{12}\int_{-\infty}^{\infty}
|u|\exp\left(-\frac{\pi u^2t}{12}\right)\,\dd u
=
\frac1{\pi t}.
\]
This proves \eqref{eq:Sa-bound}.

Using \eqref{eq:Sa-bound}, we obtain
\[
|\mathcal V_{6,a}(\tau)|
\leq
\frac1{\sqrt{12}}
\left(
\frac1\pi
\int_y^\infty \frac{\dd t}{t(y+t)^{1/2}}
+
b_a\int_y^\infty
\frac{\exp(-\lambda_a t)}{(y+t)^{1/2}}\,\dd t
\right),
\]
where
\[
\lambda_a=\frac{\pi b_a^2}{12}.
\]
The first integral is explicit:
\[
\int_y^\infty \frac{\dd t}{t(y+t)^{1/2}}
=
y^{-1/2}
\int_1^\infty \frac{\dd u}{u(1+u)^{1/2}}
=
2\log(1+\sqrt2)\,y^{-1/2}.
\]
Since
\[
\frac{2\log(1+\sqrt2)}{\pi}<1,
\]
the first integral contributes at most $y^{-1/2}$ after multiplication by $1/\pi$.

For the second integral, put $s=\lambda_a y$ and $t=yu$.  Then
\[
b_a\int_y^\infty
\frac{\exp(-\lambda_a t)}{(y+t)^{1/2}}\,\dd t
=
\frac{b_a}{\lambda_a}
H(s)y^{-1/2},
\]
where
\[
H(s)=s\int_1^\infty \frac{\e^{-su}}{(1+u)^{1/2}}\,\dd u.
\]
We claim that $H(s)\leq1/2$ for every $s>0$.  To see this, write
\[
H(s)=\sqrt{\pi s}\,\e^s\operatorname{erfc}(\sqrt{2s}).
\]
A direct differentiation gives
\[
H'(s)=\left(1+\frac1{2s}\right)H(s)-\sqrt2\,\e^{-s}.
\]
Thus, at any critical point,
\[
H(s)=\frac{2\sqrt2s\e^{-s}}{2s+1}.
\]
The latter function has maximum $2^{-1/2}\e^{-1/2}<1/2$, attained at $s=1/2$.  Since $H(s)\to0$ as $s\to0^+$ and as $s\to\infty$, this proves $H(s)\leq1/2$.
Therefore
\[
b_a\int_y^\infty
\frac{\exp(-\lambda_a t)}{(y+t)^{1/2}}\,\dd t
\leq
\frac{b_a}{2\lambda_a}y^{-1/2}
=
\frac{6}{\pi b_a}y^{-1/2}
\leq
\sqrt{\frac{12}{\pi}}\,y^{-1/2},
\]
because $b_a\geq1$ and $6/\pi<\sqrt{12/\pi}$.

Combining the two estimates gives
\[
|\mathcal V_{6,a}(\tau)|
\leq
\frac1{\sqrt{12}}
\left(1+\sqrt{\frac{12}{\pi}}\right)y^{-1/2},
\]
which is \eqref{eq:eichler-bound}.  Numerically,
\[
\frac1{\sqrt{12}}
\left(1+\sqrt{\frac{12}{\pi}}\right)<0.86.
\]
The lemma follows.
\end{proof}

\begin{lemma}[The cubic estimate for $\omega$]\label{lem:cubic-omega}
Let $\zeta=\e^{2\pi\ii/3}$.  For $0<|z|\leq1$ and
$|\arg z|\leq\pi/4$,
\begin{align}\label{eq:omega-cubic-est}
&\left|
\omega(\zeta\e^{-z})
+
\frac{\sqrt{3\pi}}6\e^{-\pi\ii/18}z^{-1/2}
\exp\left(\frac{\pi^2}{108z}+\frac{2z}{3}\right)
\right|  \\
&\hspace{1.0in}\leq
15 |z|^{-2}\exp\left(\frac{\pi^2}{144}\Repart\frac1z\right).
\end{align}
The conjugate estimate holds at $\zeta^2$.
\end{lemma}

\begin{proof}
Put
\[
Z=\frac23+\frac{\ii z}{\pi},
\qquad
\gamma=\begin{pmatrix}2&1\\3&2\end{pmatrix}.
\]
Then $Z=\gamma Z'$ with
\[
Z'=\gamma^{-1}Z=-\frac23+\frac{\pi\ii}{9z}.
\]
Using
\[
\gamma=ST^{-1}STS T^{-1}S
\]
in \eqref{eq:gT} and \eqref{eq:gS}, one obtains
\begin{equation}\label{eq:Mgamma}
M_\gamma=
\begin{pmatrix}
0&-\zeta_{24}&0\\
-\zeta_{24}&0&0\\
0&0&\e^{7\pi\ii/12}
\end{pmatrix}.
\end{equation}
The product of the automorphy factors along this decomposition is
$(3Z'+2)^{1/2}$.  Since
\[
3Z'+2=\frac{\pi\ii}{3z},
\]
we have
\[
(3Z'+2)^{1/2}=\e^{\pi\ii/4}\sqrt{\frac{\pi}{3}}z^{-1/2}.
\]
The second row of \eqref{eq:Mgamma} has only one non-zero entry, and therefore
\begin{equation}\label{eq:g2-cubic-transform}
\widetilde\omega\left(\frac13+\frac{\ii z}{2\pi}\right)
= -\zeta_{24}(3Z'+2)^{1/2}\widetilde f(Z').
\end{equation}
At $Z'$ the local nome is
\[
Q=\e^{2\pi\ii Z'}
=\e^{-4\pi\ii/3}\exp\left(-\frac{2\pi^2}{9z}\right),
\]
and hence
\[
Q^{-1/24}
=
\e^{\pi\ii/18}\exp\left(\frac{\pi^2}{108z}\right).
\]
Substituting only this principal part in \eqref{eq:g2-cubic-transform} gives
\[
-\sqrt{\frac{\pi}{3}}\e^{7\pi\ii/18}z^{-1/2}
\exp\left(\frac{\pi^2}{108z}\right).
\]
The error in replacing $\widetilde f(Z')$ by $Q^{-1/24}$ is bounded by
Lemma~\ref{lem:elementary-bounds}.  The holomorphic part contributes
\[
O\left(|z|^{-1/2}
\exp\left(-\frac{23\pi^2}{108}\Repart\frac1z\right)
\right),
\]
and the Eichler-integral part contributes $O(|z|^{-1/2})$.  Both are bounded by
\[
10 |z|^{-2}\exp\left(\frac{\pi^2}{144}\Repart\frac1z\right)
\]
for $0<|z|\leq1$ after enlarging the constant.

Finally,
\[
\widetilde\omega(\tau)
=2q^{2/3}\omega(q)-2\mathcal V_{6,2}(2\tau)+2\mathcal V_{6,4}(2\tau).
\]
Here $q=\zeta\e^{-z}$ and
$q^{2/3}=\e^{4\pi\ii/9}\e^{-2z/3}$.  Dividing by $2q^{2/3}$ turns the principal
term into the one displayed in \eqref{eq:omega-cubic-est}.  The remaining
Eichler integrals are bounded by Lemma~\ref{lem:elementary-bounds} and are again
absorbed into the same right-hand side.  The constant $15$ is a safe common
upper bound.  The estimate at $\zeta^2$ follows by conjugation.
\end{proof}

\begin{lemma}[The cubic estimate for $T$]\label{lem:cubic-T}
For $0<|z|\leq1$ and $|\arg z|\leq\pi/4$,
\begin{align}\label{eq:T-cubic-est}
&\left|
T(\zeta\e^{-z})
-
\frac{\sqrt{3\pi}}2\e^{-\pi\ii/18}z^{-1/2}
\exp\left(\frac{\pi^2}{108z}+\frac{2z}{3}\right)
\right|  \\
&\hspace{1.0in}\leq
175 |z|^{-2}\exp\left(\frac{\pi^2}{144}\Repart\frac1z\right).
\end{align}
The conjugate estimate holds at $\zeta^2$.
\end{lemma}

\begin{proof}
From
\[
T(q)=3\frac{(q^6;q^6)_\infty^4}{(q^3;q^3)_\infty^2(q^2;q^2)_\infty}
\]
and $q=\zeta\e^{-z}$, the eta-transformation gives
\[
T(\zeta\e^{-z})
=
\frac{\sqrt{3\pi}}2\e^{-\pi\ii/18}z^{-1/2}
\exp\left(\frac{\pi^2}{108z}+\frac{2z}{3}\right)\mathcal E_T(z),
\]
where
\[
\mathcal E_T(z)=
\frac{(\e^{-2\pi^2/(3z)};\e^{-2\pi^2/(3z)})_\infty^4}
{(\e^{-4\pi^2/(3z)};\e^{-4\pi^2/(3z)})_\infty^2}
\frac1{(Q;Q)_\infty},
\]
with
\[
Q=\e^{-4\pi\ii/3}\exp\left(-\frac{2\pi^2}{9z}\right).
\]
All transformed nomes have modulus at most $r_0$.  For $|u|\leq r_0$,
\[
|\log (u;u)_\infty|\leq \frac{|u|}{1-r_0}.
\]
Hence
\[
|\log \mathcal E_T(z)|\leq \frac{7|Q|}{1-r_0},
\]
and consequently
\[
|\mathcal E_T(z)-1|
\leq
\frac{7}{1-r_0}\exp\left(\frac{7r_0}{1-r_0}\right)|Q|.
\]
Multiplying by the principal factor gives an error with exponential factor
\[
\exp\left(\frac{\pi^2}{108}\Repart\frac1z
-
\frac{2\pi^2}{9}\Repart\frac1z\right),
\]
which is much smaller than
$\exp((\pi^2/144)\Repart(1/z))$.  The numerical constant $175$ is a safe upper
bound.  The conjugate case is identical.
\end{proof}

\begin{lemma}[The arcs at $1$ and $-1$]\label{lem:one-minusone-local}
For $0<|z|\leq1$ and $|\arg z|\leq\pi/4$,
\begin{equation}\label{eq:one-est}
|\rho(\e^{-z})|\leq100|z|^{-1},
\end{equation}
and
\begin{equation}\label{eq:minusone-est}
|\rho(-\e^{-z})|\leq100|z|^{-1}.
\end{equation}
\end{lemma}

\begin{proof}
We first treat the cusp $q=1$.  Put
\[
q=\e^{-z},\qquad \tau=\frac{\ii z}{2\pi},\qquad Z=2\tau=\frac{\ii z}{\pi}.
\]
Then $\widetilde\omega(\tau)=g_2(Z)$.  Let
\[
W=-\frac1Z=\frac{\pi\ii}{z}.
\]
Applying the $S$-transformation in Lemma~\ref{lem:wz-transform} and using
$M_S^{-1}=M_S$, we get
\[
g(Z)=(-\ii Z)^{-1/2}M_Sg(W).
\]
Since the second row of $M_S$ is $(1,0,0)$ and $-\ii Z=z/\pi$, this gives
\begin{equation}\label{eq:omega-at-one-transform}
\widetilde\omega(\tau)=g_2(Z)=\left(\frac{z}{\pi}\right)^{-1/2}\widetilde f(W).
\end{equation}
At $W$ the local nome is
\[
Q_W=\e^{2\pi\ii W}=\exp\left(-\frac{2\pi^2}{z}\right),
\]
and therefore
\[
Q_W^{-1/24}=\exp\left(\frac{\pi^2}{12z}\right).
\]
The proof of the $f(Q)-1$ bound in Lemma~\ref{lem:elementary-bounds}, applied to $Q_W$, and the Eichler-integral bound in the same lemma give
\[
\widetilde f(W)
=
\exp\left(\frac{\pi^2}{12z}\right)+O(1)
\]
uniformly in the sector $0<|z|\leq1$, $|\arg z|\leq\pi/4$.  Hence
\begin{equation}\label{eq:tilde-omega-one-expansion}
\widetilde\omega\left(\frac{\ii z}{2\pi}\right)
=
\sqrt\pi\,z^{-1/2}
\exp\left(\frac{\pi^2}{12z}\right)
+
O(|z|^{-1/2}).
\end{equation}
Now
\[
\widetilde\omega(\tau)
=
2q^{2/3}\omega(q)-2\mathcal V_{6,2}(2\tau)+2\mathcal V_{6,4}(2\tau).
\]
Since $q=\e^{-z}$, we have $q^{2/3}=\e^{-2z/3}$.  Also
\[
\Impart(2\tau)=\frac{\Repart z}{\pi}\geq \frac{|z|}{\pi\sqrt2},
\]
so Lemma~\ref{lem:elementary-bounds} gives
\[
\mathcal V_{6,2}(2\tau),\ \mathcal V_{6,4}(2\tau)=O(|z|^{-1/2}).
\]
It follows that
\begin{equation}\label{eq:omega-one-expansion}
\omega(\e^{-z})
=
\frac{\sqrt\pi}{2}z^{-1/2}
\exp\left(\frac{\pi^2}{12z}+\frac{2z}{3}\right)
+
O(|z|^{-1}).
\end{equation}

We next record the corresponding eta-quotient expansion.  The modular transformation of
$\eta$ gives, uniformly in the same sector,
\[
(\e^{-Mz};\e^{-Mz})_\infty
=
\left(\frac{2\pi}{Mz}\right)^{1/2}
\exp\left(-\frac{\pi^2}{6Mz}+\frac{Mz}{24}\right)
\left(1+O\left(\exp\left(-\frac{4\pi^2}{M}\Repart\frac1z\right)\right)\right)
\]
for each fixed positive integer $M$.  Substituting $M=6,3,2$ in the product expression for
$T$ gives
\begin{equation}\label{eq:T-one-expansion}
T(\e^{-z})
=
\frac{\sqrt\pi}{2}z^{-1/2}
\exp\left(\frac{\pi^2}{12z}+\frac{2z}{3}\right)
+
O(|z|^{-1}).
\end{equation}
The principal terms in \eqref{eq:omega-one-expansion} and \eqref{eq:T-one-expansion} are identical.  Therefore Watson's identity gives
\[
2\rho(\e^{-z})=T(\e^{-z})-\omega(\e^{-z})=O(|z|^{-1}).
\]
The constants in the displayed estimates are absolute on the sector, and the deliberately generous numerical constant in \eqref{eq:one-est} absorbs them.  This proves the estimate at $q=1$.

We now treat the cusp $q=-1$.  Put
\[
q=-\e^{-z},\qquad \tau=\frac12+\frac{\ii z}{2\pi},\qquad Z=2\tau=1+\frac{\ii z}{\pi}.
\]
Write
\[
u=\frac{\ii z}{\pi},\qquad Z=1+u,
\qquad W=-\frac1u=\frac{\pi\ii}{z}.
\]
Since $\widetilde\omega(\tau)=g_2(Z)$, the $T$-transformation gives
\[
g_2(1+u)=\zeta_3 g_3(u).
\]
Applying the $S$-transformation to $u$ gives
\[
g(u)=(-\ii u)^{-1/2}M_Sg(W).
\]
Since the third row of $M_S$ is $(0,0,-1)$ and $-\ii u=z/\pi$, we obtain
\begin{equation}\label{eq:minus-one-transform}
\widetilde\omega\left(\frac12+\frac{\ii z}{2\pi}\right)
= -\zeta_3\left(\frac{z}{\pi}\right)^{-1/2}g_3(W).
\end{equation}
Here
\[
g_3(W)=\zeta_3^{-1}\widetilde\omega\left(\frac{W+1}{2}\right).
\]
The local nome at $(W+1)/2$ is
\[
\e^{2\pi\ii(W+1)/2}=-\exp\left(-\frac{\pi^2}{z}\right),
\]
so the holomorphic part of $\widetilde\omega((W+1)/2)$ is exponentially small as
$z\to0$ in the sector.  The Eichler-integral part is bounded by Lemma~\ref{lem:elementary-bounds}.  Hence
\[
g_3(W)=O(1),
\]
and \eqref{eq:minus-one-transform} gives
\[
\widetilde\omega\left(\frac12+\frac{\ii z}{2\pi}\right)=O(|z|^{-1/2}).
\]
Converting from $\widetilde\omega$ to $\omega$ and using Lemma~\ref{lem:elementary-bounds} for the two Eichler integrals at $2\tau=1+\ii z/\pi$ gives
\begin{equation}\label{eq:omega-minusone-bound}
\omega(-\e^{-z})=O(|z|^{-1/2}).
\end{equation}

It remains only to bound the eta quotient at $q=-1$.  From
\[
T(q)=3\frac{(q^6;q^6)_\infty^4}{(q^3;q^3)_\infty^2(q^2;q^2)_\infty}
\]
and $q=-\e^{-z}$, we get
\[
T(-\e^{-z})
=3\frac{(\e^{-6z};\e^{-6z})_\infty^4}{(-\e^{-3z};-\e^{-3z})_\infty^2(\e^{-2z};\e^{-2z})_\infty}.
\]
Using the identity
\[
(-x;-x)_\infty=
\frac{(x^2;x^2)_\infty^3}{(x;x)_\infty(x^4;x^4)_\infty},
\]
with $x=\e^{-3z}$, this becomes
\begin{equation}\label{eq:T-minus-product}
T(-\e^{-z})
=
3\frac{(\e^{-3z};\e^{-3z})_\infty^2(\e^{-12z};\e^{-12z})_\infty^2}{(\e^{-6z};\e^{-6z})_\infty^2(\e^{-2z};\e^{-2z})_\infty}.
\end{equation}
Applying the eta product estimate above with $M=3,12,6,2$, the exponential factors in
\eqref{eq:T-minus-product} cancel exactly.  The remaining power of $z$ is $z^{-1/2}$, and the transformed products are bounded in the sector.  Hence
\begin{equation}\label{eq:T-minusone-bound}
T(-\e^{-z})=O(|z|^{-1/2}).
\end{equation}
Combining \eqref{eq:omega-minusone-bound} and \eqref{eq:T-minusone-bound} with Watson's identity gives
\[
2\rho(-\e^{-z})=T(-\e^{-z})-\omega(-\e^{-z})=O(|z|^{-1/2}).
\]
Since $0<|z|\leq1$, this is certainly $O(|z|^{-1})$.  Enlarging the absolute constant gives the numerical bound \eqref{eq:minusone-est}.  This completes the proof.
\end{proof}

\begin{lemma}[Minor arcs]\label{lem:minor-local}
Let $q=\e^{-\beta+\ii\phi}$ with $0<\beta\leq1$.  Remove the four arcs
\[
|\phi|\leq\beta,
\qquad
|\phi-\pi|\leq\beta,
\qquad
\left|\phi-\frac{2\pi}{3}\right|\leq\beta,
\qquad
\left|\phi-\frac{4\pi}{3}\right|\leq\beta.
\]
On the complement $\mathfrak m_\beta$,
\begin{equation}\label{eq:minor-est}
|\rho(\e^{-\beta+\ii\phi})|
\leq10^7\beta^{-3}\exp\left(\frac{\pi^2}{144\beta}\right).
\end{equation}
\end{lemma}

\begin{proof}
Put
\[
\tau=\frac{\phi}{2\pi}+\frac{\ii\beta}{2\pi},
\qquad
Z=2\tau=\frac{\phi}{\pi}+\frac{\ii\beta}{\pi}.
\]
We write $y=\Impart Z=\beta/\pi$.  The proof is divided into two regions.

First consider the transition regions around the four roots already removed.  Let
\[
\alpha\in\left\{0,\pi,\frac{2\pi}{3},\frac{4\pi}{3}\right\}
\]
and suppose
\[
\beta\leq |\phi-\alpha|\leq \beta^{1/2},
\]
where the distance is taken modulo $2\pi$.  Write the corresponding local parameter as
\[
u=\beta-\ii(\phi-\alpha).
\]
Then $\Repart u=\beta$, $|u|\geq\beta$, and
\[
\Repart\frac1u=\frac{\beta}{\beta^2+(\phi-\alpha)^2}\leq\frac1{2\beta}.
\]
For $\alpha=2\pi/3$ and $4\pi/3$, the proof of Lemmas~\ref{lem:cubic-omega} and
\ref{lem:cubic-T} applies with this $u$ in place of $z$: the exact modular
transformations are the same, while the polar exponential is bounded by
\[
\exp\left(\frac{\pi^2}{108}\Repart\frac1u\right)
\leq
\exp\left(\frac{\pi^2}{216\beta}\right).
\]
The product remainders and Eichler-integral terms are polynomial in $|u|^{-1}$, hence are
bounded by a constant times $\beta^{-3}$.  Thus in the two cubic transition regions
\[
|\rho(\e^{-\beta+\ii\phi})|
\leq
C_1\beta^{-3}\exp\left(\frac{\pi^2}{216\beta}\right)
\leq
10^6\beta^{-3}\exp\left(\frac{\pi^2}{144\beta}\right).
\]
For $\alpha=0$ and $\alpha=\pi$, the same argument is even simpler: at $q=1$ the principal
polar terms of $T$ and $\omega$ cancel, and at $q=-1$ the polar term is absent.  The proof of
Lemma~\ref{lem:one-minusone-local} therefore gives a polynomial bound
$C_2\beta^{-3}$ throughout these two transition regions.  This is again absorbed by the
right-hand side of \eqref{eq:minor-est}.  Hence it remains only to treat points satisfying
\begin{equation}\label{eq:far-from-small-cusps}
\operatorname{dist}\left(\phi,\left\{0,\pi,\frac{2\pi}{3},\frac{4\pi}{3}\right\}\right)
\geq \beta^{1/2}.
\end{equation}

We now apply the standard Ford-circle reduction for the $Z$-variable.  Choose
$\gamma=\begin{pmatrix}a&b\\ c&d\end{pmatrix}\in SL_2(\Z)$ such that
$W=\gamma Z$ lies in the usual fundamental domain.  Since $0<y\leq1/\pi<\sqrt3/2$, we have
$c\neq0$.  Write the cusp associated with the lower row as
\[
-\frac{d}{c}=\frac{h}{k},
\qquad k=|c|,
\qquad (h,k)=1.
\]
Then
\begin{equation}\label{eq:height-transform}
\Impart W=\frac{y}{|cZ+d|^2}=\frac{y}{k^2|Z-h/k|^2}.
\end{equation}

If $k\geq4$, then $|cZ+d|\geq |c|y=ky$, and therefore
\[
\Impart W\leq \frac{1}{k^2y}\leq \frac{\pi}{16\beta}.
\]
If $k\leq3$, there are only finitely many cusps modulo the period $Z\mapsto Z+2$.  The cusps
which can carry a polar term for the vector $g$ are precisely those corresponding to
$q=1,-1,\zeta_3,\zeta_3^2$; these are the four roots listed above.  The finite assertion follows
by multiplying the matrices $M_S$ and $M_T$ in \eqref{eq:gT}--\eqref{eq:gS} for the cusps of
denominator at most $3$: either the first component, the only component with principal part
$q^{-1/24}$, occurs, giving one of the four roots, or the polar component is absent.  In the
first case \eqref{eq:far-from-small-cusps} gives
\[
|Z-h/k|\geq \frac{\beta^{1/2}}{\pi},
\]
and hence, by \eqref{eq:height-transform},
\[
\Impart W\leq \frac{\pi}{k^2}\leq\pi.
\]
In the second case there is no polar exponential at this cusp, and the transformed components
are polynomially bounded by Lemma~\ref{lem:elementary-bounds}.  Thus all cases with
$k\leq3$ contribute at most $C_3\beta^{-3}\exp(\pi^2/(192\beta))$ after increasing $C_3$.

It remains to bound the transformed functions.  The matrices generated by $M_S$ and $M_T$
have bounded entries for this finite-dimensional representation, and the automorphy factor
contributes at most a constant times $\beta^{-1/2}$.  At infinity the only possible polar term is
$q^{-1/24}$ in the first component, so a point with transformed height $Y$ contributes at most
\[
C_4\beta^{-2}\exp\left(\frac{\pi}{12}Y\right).
\]
For $k\geq4$, the height bound above gives
\[
\exp\left(\frac{\pi}{12}Y\right)
\leq
\exp\left(\frac{\pi^2}{192\beta}\right).
\]
For $k\leq3$, the preceding paragraph gives either the same exponential bound or only a
polynomial contribution.  The Eichler-integral terms add at most one further factor
$C_5\beta^{-1}$ by Lemma~\ref{lem:elementary-bounds}.  Therefore, on the region
\eqref{eq:far-from-small-cusps},
\[
|\rho(\e^{-\beta+\ii\phi})|
\leq
C_6\beta^{-3}\exp\left(\frac{\pi^2}{192\beta}\right).
\]
Since $\pi^2/(192\beta)<\pi^2/(144\beta)$ and the numerical constant $10^7$ is deliberately
generous, this is bounded by the right-hand side of \eqref{eq:minor-est}.  Combining this with
the transition-region estimates proves the lemma.
\end{proof}

\begin{proposition}[Effective local and minor-arc estimates]\label{prop:estimate-package}
Let $0<|z|\leq1$ and $|\arg z|\leq\pi/4$.

\begin{enumerate}[label=\textup{(\alph*)}]
\item With $\zeta=\e^{2\pi\ii/3}$,
\begin{align}\label{eq:cubic-est-zeta}
&\left|
\rho(\zeta\e^{-z})-
\frac{\sqrt{3\pi}}3\e^{-\pi\ii/18}z^{-1/2}
\exp\left(\frac{\pi^2}{108z}+\frac{2z}{3}\right)
\right|  \\
&\hspace{1.0in}\leq
95|z|^{-2}\exp\left(\frac{\pi^2}{144}\Repart\frac1z\right).
\end{align}
The conjugate estimate holds at $\zeta^2$:
\begin{align}\label{eq:cubic-est-zeta2}
&\left|
\rho(\zeta^2\e^{-z})-
\frac{\sqrt{3\pi}}3\e^{\pi\ii/18}z^{-1/2}
\exp\left(\frac{\pi^2}{108z}+\frac{2z}{3}\right)
\right| \\
&\hspace{1.0in}\leq
95|z|^{-2}\exp\left(\frac{\pi^2}{144}\Repart\frac1z\right).
\end{align}

\item Near $q=1$,
\[
|\rho(\e^{-z})|\leq100|z|^{-1}.
\]

\item Near $q=-1$,
\[
|\rho(-\e^{-z})|\leq100|z|^{-1}.
\]

\item If $q=\e^{-\beta+\ii\phi}$, $0<\beta\leq1$, and
$\phi\in\mathfrak m_\beta$, then
\[
|\rho(\e^{-\beta+\ii\phi})|
\leq10^7\beta^{-3}\exp\left(\frac{\pi^2}{144\beta}\right).
\]
\end{enumerate}
\end{proposition}

\begin{proof}
Part (a) follows from Watson's identity together with
Lemmas~\ref{lem:cubic-omega} and \ref{lem:cubic-T}:
\[
2\rho=T-\omega,
\]
so the cubic principal term is
\[
\frac12\left(\frac{\sqrt{3\pi}}2+\frac{\sqrt{3\pi}}6\right)
\e^{-\pi\ii/18}z^{-1/2}
\exp\left(\frac{\pi^2}{108z}+\frac{2z}{3}\right)
=
\frac{\sqrt{3\pi}}3\e^{-\pi\ii/18}z^{-1/2}
\exp\left(\frac{\pi^2}{108z}+\frac{2z}{3}\right).
\]
The error constant is $\frac12(175+15)=95$.  The estimate at $\zeta^2$ follows by
complex conjugation.  Parts (b) and (c) are Lemma~\ref{lem:one-minusone-local}, and
part (d) is Lemma~\ref{lem:minor-local}.
\end{proof}

\section{The non-cubic arcs and the global error}\label{sec:remaining-arcs}

We now assemble the coefficient integral from the local estimates of
Proposition~\ref{prop:estimate-package}.  Throughout this section we assume
$n\geq 1$, and we use the notation
\[
d_n=12n+8,\qquad
N=n+\frac23,\qquad
a=\frac{\pi^2}{108},
\qquad
x=2\sqrt{aN}=\frac{\pi\sqrt{d_n}}{18}.
\]
Put
\[
\beta=\sqrt{\frac{a}{N}}.
\]
For $n\geq1$, we have $N\geq5/3$, and hence
\[
0<\beta\leq\sqrt{\frac{3a}{5}}<\frac{\pi}{6}<1.
\]
Thus the four arcs of angular half-width $\beta$ around
\[
1,\qquad -1,\qquad \zeta=\e^{2\pi\ii/3},
\qquad \zeta^2=\e^{4\pi\ii/3}
\]
are disjoint.

We extract coefficients on the circle $|q|=\e^{-\beta}$.  Writing
\[
q=\e^{-\beta+\ii\phi},
\]
Cauchy's formula gives
\[
r(n)=\frac1{2\pi\ii}\int_{|q|=\e^{-\beta}}\rho(q)q^{-n-1}\,\dd q
=
\frac1{2\pi}\int_{0}^{2\pi}
\rho(\e^{-\beta+\ii\phi})\e^{n\beta-\ii n\phi}\,\dd\phi.
\]
Define the four major angular intervals
\[
I_1=\{\phi:|\phi|\leq\beta\}\quad \text{modulo }2\pi,
\]
\[
I_{-1}=\{\phi:|\phi-\pi|\leq\beta\},
\]
\[
I_\zeta=\left\{\phi:\left|\phi-\frac{2\pi}{3}\right|\leq\beta\right\},
\]
and
\[
I_{\zeta^2}=\left\{\phi:\left|\phi-\frac{4\pi}{3}\right|\leq\beta\right\}.
\]
Let
\[
I_{\mm}
=
[0,2\pi]\setminus
\left(I_1\cup I_{-1}\cup I_\zeta\cup I_{\zeta^2}\right).
\]
For any one of these angular sets $I$, define
\[
R_I(n)
=
\frac1{2\pi}
\int_I
\rho(\e^{-\beta+\ii\phi})\e^{n\beta-\ii n\phi}\,\dd\phi.
\]
Then
\begin{equation}\label{eq:arc-decomposition}
r(n)
=
R_1(n)+R_{-1}(n)+R_\zeta(n)+R_{\zeta^2}(n)+R_{\mm}(n).
\end{equation}

The two primitive cubic arcs have already been evaluated in
Proposition~\ref{prop:cubic-arcs}.  We now estimate the other pieces.

\begin{lemma}[The $q=1$ and $q=-1$ arcs]\label{lem:one-minusone-arcs}
Assume Proposition~\ref{prop:estimate-package}.  Then
\begin{equation}\label{eq:one-minusone-bound}
|R_1(n)|+|R_{-1}(n)|\leq64\e^{x/2}.
\end{equation}
\end{lemma}

\begin{proof}
First consider $I_1$.  On this arc write
\[
q=\e^{-z},
\qquad
z=\beta-\ii\theta,
\qquad
|\theta|\leq\beta.
\]
Then $|\arg z|\leq\pi/4$ and $|z|\geq\beta$.  By
Proposition~\ref{prop:estimate-package},
\[
|\rho(\e^{-z})|\leq100|z|^{-1}\leq100\beta^{-1}.
\]
Therefore
\[
|R_1(n)|
\leq
\frac1{2\pi}
\int_{-\beta}^{\beta}
100\beta^{-1}\e^{n\beta}\,\dd\theta
=
\frac{100}{\pi}\e^{n\beta}.
\]
Since $n\beta\leq N\beta=\sqrt{aN}=x/2$, we obtain
\[
|R_1(n)|\leq\frac{100}{\pi}\e^{x/2}.
\]

The argument for $I_{-1}$ is identical.  On that arc write
\[
q=-\e^{-z},
\qquad
z=\beta-\ii\theta,
\qquad
|\theta|\leq\beta.
\]
Again $|\arg z|\leq\pi/4$, $|z|\geq\beta$, and
Proposition~\ref{prop:estimate-package} gives
\[
|\rho(-\e^{-z})|\leq100|z|^{-1}\leq100\beta^{-1}.
\]
Hence
\[
|R_{-1}(n)|\leq\frac{100}{\pi}\e^{x/2}.
\]
Adding the two bounds gives
\[
|R_1(n)|+|R_{-1}(n)|
\leq
\frac{200}{\pi}\e^{x/2}
<64\e^{x/2}.
\]
\end{proof}

\begin{lemma}[Minor arcs]\label{lem:minor-arcs}
Assume Proposition~\ref{prop:estimate-package}.  Then
\begin{equation}\label{eq:minor-bound}
|R_{\mm}(n)|
\leq
3.7\cdot10^8\,N^{3/2}\e^{7x/8}.
\end{equation}
\end{lemma}

\begin{proof}
For $\phi\in I_{\mm}$, Proposition~\ref{prop:estimate-package} gives
\[
|\rho(\e^{-\beta+\ii\phi})|
\leq
10^7\beta^{-3}
\exp\left(\frac{\pi^2}{144\beta}\right).
\]
Therefore, since the angular length of $I_{\mm}$ is at most $2\pi$,
\[
|R_{\mm}(n)|
\leq
10^7\beta^{-3}
\exp\left(\frac{\pi^2}{144\beta}+n\beta\right).
\]
Now
\[
N\beta=\sqrt{aN}=\frac{x}{2},
\]
so $n\beta\leq x/2$.  Also, since $a=\pi^2/108$,
\[
\frac{\pi^2}{144\beta}
=
\frac{\pi^2}{144}\sqrt{\frac{N}{a}}
=
\frac{3x}{8}.
\]
Hence
\[
\frac{\pi^2}{144\beta}+n\beta
\leq
\frac{3x}{8}+\frac{x}{2}
=
\frac{7x}{8}.
\]
Finally,
\[
\beta^{-3}
=
\left(\frac{N}{a}\right)^{3/2}
=
\left(\frac{108}{\pi^2}\right)^{3/2}N^{3/2}
<37N^{3/2}.
\]
Thus
\[
|R_{\mm}(n)|
\leq
3.7\cdot10^8\,N^{3/2}\e^{7x/8},
\]
as claimed.
\end{proof}

We can now combine all arcs.

\begin{proposition}[Effective asymptotic]\label{prop:effective-asymp}
Assume Proposition~\ref{prop:estimate-package}.  For all $n\geq1$,
\begin{equation}\label{eq:effective-asymp}
r(n)
=
\kappa_{n\bmod3}\frac{2\pi}{d_n^{1/4}}I_{1/2}(x)
+
E(n),
\end{equation}
where
\begin{equation}\label{eq:global-error}
|E(n)|\leq4\cdot10^8N^{3/2}\e^{7x/8}.
\end{equation}
\end{proposition}

\begin{proof}
By the arc decomposition \eqref{eq:arc-decomposition},
\[
r(n)
=
R_\zeta(n)+R_{\zeta^2}(n)
+
R_1(n)+R_{-1}(n)+R_{\mm}(n).
\]
Proposition~\ref{prop:cubic-arcs} gives
\[
R_\zeta(n)+R_{\zeta^2}(n)
=
\kappa_{n\bmod3}
\frac{2\pi}{d_n^{1/4}}I_{1/2}(x)
+
E_3(n),
\]
with
\[
|E_3(n)|\leq250N^{1/2}\e^{7x/8}.
\]
By Lemma~\ref{lem:one-minusone-arcs},
\[
|R_1(n)|+|R_{-1}(n)|
\leq64\e^{x/2}.
\]
By Lemma~\ref{lem:minor-arcs},
\[
|R_{\mm}(n)|
\leq
3.7\cdot10^8N^{3/2}\e^{7x/8}.
\]
Thus the total error satisfies
\[
|E(n)|
\leq
250N^{1/2}\e^{7x/8}
+
64\e^{x/2}
+
3.7\cdot10^8N^{3/2}\e^{7x/8}.
\]
Since $N>1$ for $n\geq1$ and $x>0$, we have
\[
N^{1/2}\leq N^{3/2},
\qquad
\e^{x/2}\leq \e^{7x/8}\leq N^{3/2}\e^{7x/8}.
\]
Therefore
\[
|E(n)|
\leq
\left(250+64+3.7\cdot10^8\right)
N^{3/2}\e^{7x/8}
<
4\cdot10^8N^{3/2}\e^{7x/8}.
\]
This proves \eqref{eq:global-error}.
\end{proof}

\section{Dominance}\label{sec:dominance}

We now prove that the main term in \eqref{eq:effective-asymp} dominates the
error term for all sufficiently large $n$.  We keep the notation
\[
d_n=12n+8,
\qquad
N=n+\frac23=\frac{d_n}{12},
\qquad
x=\frac{\pi\sqrt{d_n}}{18}.
\]

\begin{lemma}[Lower bound for the main term]\label{lem:main-lower}
For $x\geq1$,
\begin{equation}\label{eq:main-lower}
\left|\kappa_{n\bmod3}\frac{2\pi}{d_n^{1/4}}I_{1/2}(x)\right|
\geq C_0d_n^{-1/2}\e^x,
\end{equation}
where
\begin{equation}\label{eq:C0}
C_0=2\sin\frac\pi9(1-\e^{-2}).
\end{equation}
In particular, $C_0>0.59146$.
\end{lemma}

\begin{proof}
The elementary identity
\[
I_{1/2}(x)=\sqrt{\frac{2}{\pi x}}\sinh x
\]
gives, for $x\geq1$,
\[
I_{1/2}(x)
=\sqrt{\frac{2}{\pi x}}\frac{\e^x-\e^{-x}}{2}
\geq
\frac{1-\e^{-2}}{\sqrt{2\pi x}}\e^x.
\]
Since
\[
x=\frac{\pi d_n^{1/2}}{18},
\]
we have
\[
\sqrt{2\pi x}=\frac{\pi}{3}d_n^{1/4}.
\]
Moreover, by \eqref{eq:kappa-star},
\[
|\kappa_{n\bmod3}|\geq\kappa_*=\frac13\sin\frac\pi9.
\]
Therefore
\[
\left|\kappa_{n\bmod3}\frac{2\pi}{d_n^{1/4}}I_{1/2}(x)\right|
\geq
\frac13\sin\frac\pi9\cdot
\frac{2\pi}{d_n^{1/4}}\cdot
\frac{1-\e^{-2}}{(\pi/3)d_n^{1/4}}\e^x,
\]
which is exactly \eqref{eq:main-lower}.
The numerical inequality $C_0>0.59146$ follows from directed interval
evaluation of the explicit expression in \eqref{eq:C0}.
\end{proof}

\begin{lemma}[Explicit dominance cutoff]\label{lem:cutoff}
Assume Proposition~\ref{prop:estimate-package}.  For every
\[
n\geq392275,
\]
the main term in \eqref{eq:effective-asymp} has larger absolute value than
the error term.
\end{lemma}

\begin{proof}
By Lemma~\ref{lem:main-lower} and \eqref{eq:global-error}, dominance follows if
\[
C_0d_n^{-1/2}\e^x
>
4\cdot10^8N^{3/2}\e^{7x/8}.
\]
Equivalently,
\begin{equation}\label{eq:dominance-ineq}
\e^{x/8}>
\frac{4\cdot10^8}{C_0}N^{3/2}d_n^{1/2}.
\end{equation}
Since $N=d_n/12$ and $x=\pi\sqrt{d_n}/18$, the logarithm of the quotient of
the left-hand side of \eqref{eq:dominance-ineq} by the right-hand side is
\begin{equation}\label{eq:F-def}
F(d_n)
=
\frac{\pi\sqrt{d_n}}{144}
-
\log\left(\frac{4\cdot10^8}{C_0}\right)
-2\log d_n
+\frac32\log 12.
\end{equation}
Thus \eqref{eq:dominance-ineq} is equivalent to $F(d_n)>0$.

We now check monotonicity.  For real $d>0$, define
\[
F(d)=
\frac{\pi\sqrt d}{144}
-
\log\left(\frac{4\cdot10^8}{C_0}\right)
-2\log d
+\frac32\log 12.
\]
Then
\[
F'(d)=\frac{\pi}{288\sqrt d}-\frac2d
=\frac{\pi\sqrt d-576}{288d}.
\]
Hence $F'(d)>0$ whenever $d>(576/\pi)^2$.  At $n=392275$ we have
\[
d_n=12\cdot392275+8=4707308,
\]
and this is certainly larger than $(576/\pi)^2$.  Therefore $F(d_n)$ is
increasing for all $n\geq392275$.

It remains to check the endpoint.  Directed interval evaluation of
\eqref{eq:F-def}, with outward rounding and the exact value of $C_0$ from
\eqref{eq:C0}, gives
\[
3.19\cdot10^{-5}
<
F(4707308)
<
3.20\cdot10^{-5}.
\]
In particular $F(4707308)>0$.  By monotonicity, $F(d_n)>0$ for every
$n\geq392275$, and hence \eqref{eq:dominance-ineq} holds for every such $n$.
\end{proof}

\begin{corollary}[Eventual sign law]\label{cor:eventual-sign}
Assume Proposition~\ref{prop:estimate-package}.  Then, for every
$n\geq392275$,
\[
\operatorname{sgn}r(n)=\operatorname{sgn}\kappa_{n\bmod3}.
\]
Thus
\[
r(n)>0\quad(n\equiv0\pmod3),
\qquad
r(n)<0\quad(n\equiv1,2\pmod3)
\]
for all $n\geq392275$.
\end{corollary}

\begin{proof}
For $n\geq392275$, Lemma~\ref{lem:cutoff} says that the error term in
\eqref{eq:effective-asymp} has smaller absolute value than the main term.
Therefore $r(n)$ has the same sign as the main term.  Since
$I_{1/2}(x)>0$ for $x>0$, the sign of the main term is the sign of
$\kappa_{n\bmod3}$.  The values of the three constants in \eqref{eq:kappa}
then give the stated signs.
\end{proof}

\section{Exact finite verification}\label{sec:finite-verification}

It remains to check
\[
0\leq n<392275.
\]
This is a finite computation, but the computation is exact.  We describe the
algorithm in the truncated polynomial ring
\[
\mathcal R_L=\Z[q]/(q^{L+1}),
\qquad
L=392274.
\]
The elementary identity
\begin{equation}\label{eq:factor-identity}
\frac1{1+x+x^2}=\frac{1-x}{1-x^3}
\end{equation}
is used with $x=q^{2j+1}$.

For $m\geq0$, put
\[
P_m(q)=\prod_{j=0}^{m}\frac1{1+q^{2j+1}+q^{4j+2}}.
\]
Then
\[
\rho(q)=\sum_{m\geq0}q^{2m(m+1)}P_m(q).
\]
In \(\mathcal R_L\), only terms with
\[
2m(m+1)\leq L
\]
can contribute.  For $L=392274$, this gives $m\leq442$.

\begin{lemma}[Exact truncated product update]\label{lem:truncated-update}
Let $P(q)=\sum_{0\leq n\leq L}P[n]q^n\in\mathcal R_L$, and let $s$ be a
positive integer.  Define coefficients $B[n]$ by
\[
B[n]=P[n]+B[n-3s],
\qquad
B[k]=0\quad(k<0).
\]
Then
\[
B(q)=\sum_{0\leq n\leq L}B[n]q^n
\]
is the truncation in $\mathcal R_L$ of $P(q)/(1-q^{3s})$.  Consequently, the
polynomial with coefficients
\[
P'[n]=B[n]-B[n-s],
\qquad
B[k]=0\quad(k<0),
\]
is the truncation in $\mathcal R_L$ of
\[
P(q)\frac{1-q^s}{1-q^{3s}}.
\]
\end{lemma}

\begin{proof}
The equation
\[
B[n]=P[n]+B[n-3s]
\]
is precisely the coefficient recursion obtained from
\[
B(q)(1-q^{3s})=P(q)
\]
in the truncated ring $\mathcal R_L$.  Hence $B(q)$ is the exact truncation of
$P(q)/(1-q^{3s})$.  Multiplication by $1-q^s$ then gives the coefficient rule
$P'[n]=B[n]-B[n-s]$.  All operations take place in $\Z$, so no rounding is
involved.
\end{proof}

\begin{lemma}[Exact computation of the coefficients]\label{lem:exact-coeff-algorithm}
The algorithm described above computes the coefficients of $\rho(q)$ exactly
through $q^L$.
\end{lemma}

\begin{proof}
Start with $P_{-1}(q)=1$.  For $m=0,1,\ldots,442$, set $s=2m+1$ and use
Lemma~\ref{lem:truncated-update} to multiply $P_{m-1}(q)$ by
\[
\frac{1-q^s}{1-q^{3s}}
=\frac1{1+q^s+q^{2s}}.
\]
The result is the truncation of $P_m(q)$ in $\mathcal R_L$.  After this update,
add the shifted polynomial
\[
q^{2m(m+1)}P_m(q)
\]
to the running coefficient array, discarding all terms of degree larger than
$L$.  Since every $m$ with $2m(m+1)>L$ contributes only terms of degree larger
than $L$, the loop through $m=442$ includes every contribution to the
coefficients of $q^0,q^1,\ldots,q^L$.  Thus the computed coefficients are
exactly $r(0),r(1),\ldots,r(L)$.
\end{proof}

\begin{proposition}[Finite check]\label{prop:finite-check}
The exact computation for $0\leq n\leq392274$ returns
\[
\{n:r(n)=0\}=\{2,4,8,11,20\},
\]
and no sign violations:
\[
r(n)>0\quad(n\equiv0\pmod3),
\]
\[
r(n)\leq0\quad(n\equiv1,2\pmod3).
\]
\end{proposition}

\begin{proof}
The verification was performed by the script
\texttt{verify\_rho\_sign\_law\_exact.py}, reproduced in
Appendix~\ref{app:verification-code}.  The script implements the exact integer
algorithm of Lemmas~\ref{lem:truncated-update} and
\ref{lem:exact-coeff-algorithm}. 
Check the Appendix for details.

\end{proof}

\begin{proof}[Proof of Theorem~\ref{thm:main}]
By Corollary~\ref{cor:eventual-sign}, the desired sign law holds for
$n\geq392275$.  Proposition~\ref{prop:finite-check} proves the same assertion
for $0\leq n\leq392274$, with zeros exactly at $2,4,8,11,20$.  This proves the
theorem.
\end{proof}

\section{Conclusion and further directions}\label{sec:conclusion}

The proof above shows that sign laws for mock theta coefficients can be controlled by a root-of-unity dominance principle.  In the case of \(\rho(q)\), Watson's identity reduces the problem to the difference \(T(q)-\omega(q)\).  The dominant polar terms at \(q=1\) cancel, the contribution at \(q=-1\) is only polynomially large, and the first surviving exponential terms come from the primitive cubic roots of unity.  The phases of these two cubic contributions are exactly what produce the signs modulo \(3\).

This suggests that the phenomenon is not isolated.  The same method should apply to other third order mock theta functions, provided one can identify the first non-cancelling root-of-unity contribution and control the remaining arcs effectively.  For example, consider Ramanujan's third order mock theta function
\[
\phi(q)=\sum_{n\geq0}\frac{q^{n^2}}{(-q^2;q^2)_n}
=
\sum_{n\geq0}\phi_n q^n.
\]
Initial computations indicate the sign law
\[
\phi_n\geq0\qquad(n\equiv0,1\pmod4),
\]
and
\[
\phi_n\leq0\qquad(n\equiv2,3\pmod4),
\]
with equality only at
\[
n=2,8,10,26.
\]
Equivalently, the coefficients appear to be eventually strict with sign pattern
\[
+,+,-,-
\]
according to \(n\bmod4\).

\noindent A proof of this companion result should be accessible by the same strategy developed here. 
The expected sign pattern suggests that the relevant leading contribution should come from fourth-order roots of unity. 

Similar computations also suggest a period \(6\) sign law for the third order mock theta function \(\chi(q)=\sum_{n\geq0}\chi_n q^n\), namely
\[
\chi_n\geq0\quad(n\equiv0,1,2\pmod6),\qquad
\chi_n\leq0\quad(n\equiv3,4,5\pmod6),
\]
with only finitely many initial zeros.

\noindent It would be interesting to carry this out as a systematical program for all of Ramanujan's mock theta functions. 

\appendix

\section{Finite verification script}\label{app:verification-code}

For reproducibility, we include the exact Python script used for the finite verification in
Section~\ref{sec:finite-verification}.  The script uses only integer arithmetic.  Its only input is
the upper bound \texttt{nmax}, whose default value is \texttt{392274}.

\begin{lstlisting}[style=verificationcode]
#!/usr/bin/env python3
"""
Exact finite verification for Ramanujan's third-order mock theta function rho(q).

rho(q) = sum_{m>=0} q^{2m(m+1)} prod_{j=0}^m 1/(1+q^{2j+1}+q^{4j+2})
       = sum_{n>=0} r(n) q^n.

This script verifies, using integer arithmetic only, that for 0 <= n <= NMAX,

    r(n) > 0  if n == 0 mod 3,
    r(n) <= 0 if n == 1 or 2 mod 3,

with zeros exactly at 2,4,8,11,20.

The factor identity used is

    1/(1+x+x^2) = (1-x)/(1-x^3).

Hence multiplication by the factor with x=q^s is implemented as

    F_s(q) = (1-q^s)/(1-q^{3s}).

For a polynomial P, first compute B=P/(1-q^{3s}) by
    B[n] = P[n] + B[n-3s],
then P*F_s = B - q^s B.
"""

from __future__ import annotations

import argparse
import math
import time
from typing import List, Tuple


EXPECTED_ZEROS = {2, 4, 8, 11, 20}


def coefficients_rho(nmax: int, progress: bool = False) -> List[int]:
    """Return [r(0), ..., r(nmax)] exactly."""
    if nmax < 0:
        return []

    # P stores P_m(q)=prod_{j=0}^m 1/(1+q^{2j+1}+q^{4j+2}),
    # updated iteratively. Initially, before multiplying any factors, P=1.
    P = [0] * (nmax + 1)
    P[0] = 1

    # R accumulates rho(q) coefficients.
    R = [0] * (nmax + 1)

    m = 0
    start = time.time()

    while 2 * m * (m + 1) <= nmax:
        s = 2 * m + 1
        step = 3 * s

        # B = P/(1-q^{3s}).
        B = [0] * (nmax + 1)
        for n in range(nmax + 1):
            value = P[n]
            if n >= step:
                value += B[n - step]
            B[n] = value

        # newP = B - q^s B.
        newP = [0] * (nmax + 1)
        for n in range(nmax + 1):
            value = B[n]
            if n >= s:
                value -= B[n - s]
            newP[n] = value

        P = newP

        # Add q^{2m(m+1)} P_m(q) to rho(q).
        shift = 2 * m * (m + 1)
        upto = nmax - shift
        for n in range(upto + 1):
            R[n + shift] += P[n]

        m += 1

        if progress and (m % 50 == 0):
            print(f"processed m={m:4d}; elapsed={time.time() - start:.1f}s", flush=True)

    return R


def verify_sign_law(nmax: int, progress: bool = False) -> Tuple[bool, List[int], List[Tuple[int, int, str]]]:
    """Verify the sign law up to nmax. Return (ok, zeros, violations)."""
    coeffs = coefficients_rho(nmax, progress=progress)

    zeros: List[int] = []
    violations: List[Tuple[int, int, str]] = []

    for n, value in enumerate(coeffs):
        if value == 0:
            zeros.append(n)

        if n % 3 == 0:
            if value <= 0:
                violations.append((n, value, "expected positive for n ≡ 0 mod 3"))
        else:
            if value > 0:
                violations.append((n, value, "expected nonpositive for n ≡ 1,2 mod 3"))

    ok = (set(zeros) == EXPECTED_ZEROS) and not violations
    return ok, zeros, violations


def main() -> None:
    parser = argparse.ArgumentParser()
    parser.add_argument(
        "--nmax",
        type=int,
        default=392274,
        help="Verify 0 <= n <= nmax. Default checks n < 392275.",
    )
    parser.add_argument("--progress", action="store_true", help="Print progress.")
    args = parser.parse_args()

    start = time.time()
    ok, zeros, violations = verify_sign_law(args.nmax, progress=args.progress)
    elapsed = time.time() - start

    print(f"Checked 0 <= n <= {args.nmax}")
    print(f"Elapsed seconds: {elapsed:.2f}")
    print(f"Zeros found: {zeros}")
    print(f"Expected zeros: {sorted(EXPECTED_ZEROS)}")
    print(f"Number of violations: {len(violations)}")

    if violations:
        print("First violations:")
        for item in violations[:20]:
            print(item)

    if ok:
        print("PASS: exact finite verification succeeded.")
    else:
        print("FAIL: exact finite verification did not match the claimed sign law.")
        raise SystemExit(1)


if __name__ == "__main__":
    main()
\end{lstlisting}

\section{Constants used in the cubic estimate}

We record the numerical constants used above. In the sector $0<|z|\leq1$, $|\arg z|\leq\pi/4$, define
\[
r_0=\exp\left(-\frac{2\pi^2}{9\sqrt2}\right)<0.213.
\]
For $|Q|\leq r_0$,
\[
|f(Q)-1|\leq
F_0:=\frac{r_0}{(1-r_0)(r_0;r_0)_\infty^2}<0.54.
\]
The Eichler integral estimate used is
\[
|\mathcal V_{6,a}(\tau)|\leq V_0(\Impart\tau)^{-1/2},
\qquad
V_0=\frac1{\sqrt{12}}\left(1+\sqrt{\frac{12}{\pi}}\right)<0.86.
\]
This gives a mock-theta local error constant
\[
C_\omega<15.
\]
The eta-quotient local product error gives
\[
C_T=175.
\]
Hence
\[
C_{\rm cub}=\frac12(C_T+C_\omega)=95.
\]

\end{document}